\documentclass[12pt]{amsart}

\usepackage{amscd,amssymb,amsopn,amsmath,amsthm,mathrsfs,graphics,amsfonts,enumerate,verbatim,calc
}

\usepackage[all]{xy}

\usepackage{color}

\usepackage[OT2,OT1]{fontenc}
\newcommand\cyr{%
\renewcommand\rmdefault{wncyr}%
\renewcommand\sfdefault{wncyss}%
\renewcommand\encodingdefault{OT2}%
\normalfont
\selectfont}
\DeclareTextFontCommand{\textcyr}{\cyr}

\usepackage{amssymb,amsmath}

\DeclareFontFamily{OT1}{rsfs}{}
\DeclareFontShape{OT1}{rsfs}{n}{it}{<-> rsfs10}{}
\DeclareMathAlphabet{\mathscr}{OT1}{rsfs}{n}{it}

\numberwithin{equation}{section}
\newtheorem{theorem}{Theorem}[section]
\newtheorem{lemma}[theorem]{Lemma}
\newtheorem{proposition}[theorem]{Proposition}
\newtheorem{corollary}[theorem]{Corollary}

\newtheorem{question}{Question}

\theoremstyle{definition}
\newtheorem{definition}[theorem]{Definition}
\newtheorem{remark}[theorem]{Remark}
\theoremstyle{remark}

\newtheorem{example}[theorem]{Example}
\newtheorem{acknowledgement}{Acknowledgement}

\begin{document}
\title[Local cohomology of polynomial and formal power series rings]{Some results on local cohomology of polynomial and formal power series rings: the one dimensional case}

\author[Pham Hung Quy]{Pham Hung Quy}
\address{Department of Mathematics, FPT University, Hoa Lac Hi-Tech Park, Ha Noi, Viet Nam}
\email{quyph@fpt.edu.vn}

\thanks{
This work is partially supported by a fund of Vietnam National Foundation for Science
and Technology Development (NAFOSTED) under grant number
101.04-2014.25.}



\begin{abstract}
In this paper, we prove several results on the finiteness of local cohomology of polynomial and formal power series rings. In particular, we give a partial affirmative answer for a question of L. N\'{u}\~{n}ez-Betancourt in [J. Algebra 399 (2014), 770--781].
\end{abstract}

\maketitle


\section{Introduction}
The motivation of this paper is the following conjecture of G. Lyubeznik: If $R$ is a regular ring, then each local cohomology module $H^i_I(R)$ has finitely many associated prime ideals. The Lyubeznik conjecture has affirmative answers in several cases: for regular rings of prime characteristic (cf. \cite{HS93,L97}); for regular local and affine rings of characteristic zero (cf. \cite{L93}); for unramified regular local rings of mixed characteristic (cf. \cite{L00b,N13b}) and for smooth $\mathbb{Z}$-algebras (cf. \cite{BBL14}). The method of the proof of these results is considering the module structure of local cohomology over non-commutative rings, $D$-modules (resp. $F$-modules). The finiteness of these module structures (for example, finite length) yields the finiteness of $\mathrm{Ass}_S H^i_I(R)$.

Motivated by the above finiteness results, M. Hochster raised the following related
question (cf. \cite[Question 1.1]{N14}):
\begin{question}\label{Q1.1}\rm Let $(R, \frak m, k)$ be a local ring and $S$ a flat extension of $R$ with regular closed fiber. Then is
$$\mathrm{Ass}_S H^0_{\frak mS}(H^i_I(S)) = V(\frak mS) \cap \mathrm{Ass}_SH^i_I(S)$$
finite for every ideal $I \subset S$ and for every integer $i \ge 0$?
\end{question}
Suppose $S$ is a flat extension of $R$ with regular fibers. It is worth to note that if Question \ref{Q1.1} has an affirmative answer, then the finiteness conditions of $\mathrm{Ass}_SH^i_I(S)$ and $\mathrm{Ass}_RH^i_I(S)$ are equivalent. In \cite{N14}, L. N\'{u}\~{n}ez-Betancourt gave a positive answer for Question \ref{Q1.1} when $S$ is either $R[x_1, ...,x_n]$ or $R[[x_1, ...,x_n]]$ and $\dim R/(I \cap R) \le 1$. In that paper, he introduced the notion of $\Sigma$-finite $D$-modules. It should be noted that $\Sigma$-finite $D$-modules maybe not have finite length but they have finitely many associated primes. N\'{u}\~{n}ez-Betancourt asked the following question (cf. \cite[Question 5.1]{N14}).
\begin{question}\label{Q1.2}\rm Let $(R, \frak m, k)$ be a local ring and $S$ either $R[x_1, ...,x_n]$ or $R[[x_1, ...,x_n]]$. Then is $H^i_{\frak m}H^j_J(S)$ $\Sigma$-finite for every ideal $J \subset S$ and $i, j \ge 0$?
\end{question}
Throughout this paper, let $R$ be a commutative Noetherian ring and $S$ be either $R[X_1,..., X_n]$ or $R[[X_1, ..., X_n]]$. In Section 3 we modify the definition of $\Sigma$-finite $D$-modules for rings that not necessarily local rings. We prove that $H^j_J(S)$ is $\Sigma$-finite for every ideal $J \subseteq S$ satisfying $\dim R/(J \cap R) = 0$ (cf. Proposition \ref{P3.7}). Applying this result we give a positive answer for Question \ref{Q1.2} when $\dim R/(J \cap R) \le 1$ (cf. Theorem \ref{T3.8}). Moreover, a finiteness result of associated primes of local cohomology is given (cf. Corollary \ref{C3.9}).

In Section 4 we consider the following problem.
\begin{question}\label{Q1.3}\rm Suppose that $\dim R = 1$ and $S$ is either $R[X_1,..., X_n]$ or $R[[X_1, ..., X_n]]$. Is it true that $H^i_J(S)$ has only finitely many associated primes for all ideals $J$ of $S$ and all $i \ge 0$?
\end{question}
By the work of B. Bhatt et al. \cite{BBL14} Question \ref{Q1.3} has a positive answer when $S = \mathbb{Z}[x_1, ...,x_n]$. The next interesting case of Lyubeznik's conjecture is seem to be the case $S = R[x_1,...,x_n]$ with $R$ is a Dedekind domain (containing the field of rational numbers). This is a special case of Question \ref{Q1.3}. In this section we will give a partial affirmative answer of Question \ref{Q1.3} in the case $R$ contains a field of positive characteristic (cf. Proposition \ref{P4.4}). It should be noted that H. Dao and the author showed that local cohomology of Stanley-Reisner rings over a field of positive characteristic have only finitely many associated primes, see \cite{DQ16} for a more general result (see also \cite{HN16}). Finally, the readers are encouraged to \cite{R14, R16} for some results about the finiteness of associated primes of local cohomology of polynomial and power series rings over a normal domain containing a field of zero characteristic.

\section{Preliminary}
In this section we collect some basic facts on rings of differential operators and $D$-modules. Let $R$ be a Noetherian ring and $S = R[X_1, ..., X_n]$ or $S = R[[X_1, ..., X_n]]$.

\noindent {\bf Rings of differential operators.} Let $D(S, R)$ (or $D$ if there is no confusion) be the ring of $R$-linear differential operators of $S$. The ring $D(S,R)$ is defined by recursion as follows. The differential operators of order zero are
the morphisms induced by multiplying by elements in $S$. An element $\delta \in \mathrm{Hom}_R(S,S)$ is a differential operator of order less than of equal to $k+1$ if $[\delta, s]: = \delta \circ s - s \circ \delta$ is a differential operator of order less than or equal to $k$ for every $s \in S = \mathrm{Hom}_S(S,S)$. Notice that $D(S, R)$ is not a commutative ring, but $R$ is contained in the center of $D(S, R)$. In our cases $S = R[X_1, ..., X_n]$ or $S = R[[X_1, ..., X_n]]$, it is well known that (see \cite[Theorem 16.12.1]{G67})
$$D(S, R) = S\big[ \frac{1}{t!} \frac{\partial^t}{\partial x_i^t} | t \in \mathbb{N}, 1 \le i \le n \big] \subseteq \mathrm{Hom}_R(S, S).$$
{\bf Homomorphic.} Let $R'$ be another ring with $\phi : R \to R'$ a homomorphism of rings. Let $S'$ be either $R' [x_1, ..., x_n]$ or $R'[[x_1, ..., x_n]]$, respectively. Then $\phi$ induces a homomorphism between rings of differential operators $\Phi : D(S, R) \to D(S', R')$. In particular, we have a natural surjection $D(S, R) \to D(S/IS , R/I)$ for every ideal $I \subset R$.
\begin{example}[of $D$-modules] \label{E2.1}\rm
\begin{enumerate}[{(i)}]
\item It is well known that $S$ is a $D$-module.
\item Let $M$ be an $R$-module. Then $M[x_1, ..., x_n] \cong R[x_1, ..., x_n] \otimes_R M $ (resp. $R[[x_1, ..., x_n]] \otimes_R M$ and $M[[x_1, ..., x_n]]$) are $D$-modules. In particular for each $\frak m \in \mathrm{Max}(R)$ we have $(R/\frak m)[x_1, ..., x_n]$ (resp. $(R/\frak m)[[x_1, ..., x_n]]$) are $D$-modules of finite length.
\item If $M$ is a $D$-module then its localization and local cohomology of $M$ are $D$-modules.
\item In \cite{L00}, Lyubeznik defined the subcategory of the category of $D(S,R)$-modules, says $C(S, R)$, is the smallest subcategory of $D(S,R)$-modules that contains $S_f$
for all $f \in S$ and that is closed under taking submodules, quotients and extensions. In particular, the kernel, image and cokernel of a morphism of
$D(S,R)$-modules that belongs to $C(S,R)$ are also objects in $C(S,R)$. Notice that $H^{i_k}_{I_k} \cdots H^{i_1}_{I_1}(S)$ is an object in $C(S,R)$. The critical fact for the study of the finiteness of local cohomology is that every module in $C(S, R)$ has finite length as a $D$-module provided $R$ is a field (see \cite[Corollary 6]{L00}).
\end{enumerate}

\end{example}


\section{$\Sigma$-finite $D$-modules}
First, we give the definition of $\Sigma$-finite $D$-modules. Notice that we do not assume $R$ is local as \cite{N14}. Let $M$ be a $D$-module, we denote by $\mathrm{Fin}(M)$ the set of all $D$-submodules of $M$ that have finite length. Let $N$ be a $D$-module of finite length. There is a filtration of submodules $0 = N_0 \subset N_1 \subset \cdots \subset N_h = N$ such that $N_i/N_{i-1}$ is a nonzero simple $D$-module for all $i = 1, ...,h$. The factors, $N_i/N_{i-1}$, are the same, up to permutation and isomorphism, for every filtration. We denote that set of factors by $\mathcal{C}(N)$.

\begin{definition}\label{D3.1}\rm Let $M$ be a $D$-module such that $\mathrm{Supp}_R(M) \subseteq \mathrm{Max}(R)$. We say that $M$ is {\it $\Sigma$-finite} if
\begin{enumerate}[{(i)}]
\item $\bigcup_{N \in \mathrm{Fin}(M)}N = M$,
\item $\bigcup_{N \in \mathrm{Fin}(M)} \mathcal{C}(N)$ is finite, and
\item For every $N \in \mathrm{Fin}(M)$ and $L \in \mathcal{C}(N)$, $L \in C(S/\frak mS, R/\frak m)$ for some $\frak m \in \mathrm{Max}(M)$.
\end{enumerate}
 \end{definition}
If $M$ is $\Sigma$-finite, we denote $\mathcal{C}(M):=\cup_{N \in \mathrm{Fin}(M)} \mathcal{C}(N)$. It is easy to see that if
$$0 \to M' \to M \to M'' \to 0$$
is a short exact sequence of $\Sigma$-finite $D$-modules, then $\mathcal{C}(M) = \mathcal{C}(M') \cup \mathcal{C}(M'')$.

 \begin{remark}\label{R3.2}\rm If $M$ is $\Sigma$-finite then $\mathrm{Supp}_R(M)$ is a finite subset of $\mathrm{Max}(R)$. If $\mathrm{Supp}_R(M) = \{\frak m_1, ..., \frak m_r\} \subseteq \mathrm{Max}(R)$, then $M \cong \Gamma_{\frak m_1}(M) \oplus \cdots \oplus \Gamma_{\frak m_r}(M)$. Therefore all results proved in Section 3 of \cite{N14} (in the case $R$ is a local ring) can be extended for our notion of $\Sigma$-finite. For example, if $M$ is a $\Sigma$-finite $D$-module, then $H^i_J(M)$ is also a $\Sigma$-finite $D$-module for every ideal $J \subset S$ and integer $i \ge 0$.
 \end{remark}
The following give us examples of $\Sigma$-finite $D$-modules.
 \begin{lemma}\label{L3.3} Let $A$ be an Artinian $R$-module. Then $M = A \otimes_R R[x_1, ..., x_n]$ (resp. $M = A \otimes_R R[[x_1, ..., x_n]]$) is a $\Sigma$-finite $D$-module.
 \end{lemma}
 \begin{proof}
   It is easy to see that $\mathrm{Supp}_R(M) = \mathrm{Supp}_R(A)$ is a finite subset of $\mathrm{Max}(R)$. Since $A$ is Artinian, it is union of all submodules of finite length. Moreover if $L$ is an $R$-module of finite length, then $L \otimes_R S$ is a $D$-module of finite length. The assertion now follows.
 \end{proof}
 \begin{remark}\rm Suppose that $S = R[[x_1, ..., x_n]]$. In general $A \otimes_R R[[x_1, ..., x_n]] \ncong A[[x_1, ..., x_n]]$ and $A[[x_1, ..., x_n]]$ may not be $\Sigma$-finite. For example, let $R = k[t]$, where $k$ is a field and $t$ an indeterminate. Let $A = E_R(k)$ be the injective hull of $k$. Then $A \cong k[t^{-1}]$. Choose the element $a = \sum_{i = 0}^\infty t^{-i}x_1^i \in S$ we have $\mathrm{Ann}_R(a) = 0 \notin \mathrm{Max}(R)$.
 \end{remark}

\begin{lemma}\label{L3.5} Let $I$ be an ideal of $R$ such that $\dim R/I = 0$. Then $H^i_{IS}(S)$ is a $\Sigma$-finite $D$-module.
\end{lemma}
\begin{proof} We have $\dim R/I  = 0$ so $\sqrt{I} = \frak m_1 \cap \cdots \cap \frak m_r$ with $\frak m_i \in \mathrm{Max}(R)$ for all $i = 0, ..., r$. By the Mayer-Vietoris sequence we have  $H^i_I(R) \cong H^i_{\frak m_1}(R) \oplus \cdots \oplus H^i_{\frak m_r}(R)$. So $H^i_I(R)$ is Artinian for all $i \ge 0$ by \cite[Theorem 7.1.3]{BS98}. By Lemma \ref{L3.3} we have $H^i_{IS}(S) \cong H^i_I(R) \otimes_R S$ is $\Sigma$-finite.
\end{proof}

The following is very useful in the sequel.

 \begin{lemma}\label{L3.6} Let $0 \to M' \to M \to M'' \to 0$ be a short exact sequence of $D$-modules. Then
\begin{enumerate}[{(i)}]\rm
\item {\it If $M$ is $\Sigma$-finite then $M'$ and $M''$ are $\Sigma$-finite.}
\item {\it Conversely, if $M'$ and $M''$ are $\Sigma$-finite and $M'$ has finite length as a $D$-module, then $M$ is $\Sigma$-finite.}
\end{enumerate}
\end{lemma}
\begin{proof} (i) This part is \cite[Proposition 3.6]{N14}. \\
(ii) Since $M''$ is $\Sigma$-finite we have $M'' = \cup_{N'' \in \mathrm{Fin}(M'')}N''$. For each $N'' \in \mathrm{Fin}(M'')$, let $N$ be the preimage of $N''$. One can check that $N$ admits a $D$-module structure. We have the following short exact sequence of $D$-modules.
$$0 \to M' \to N \to N'' \to 0.$$
Since $M'$ has finite length as a $D$-module we have $N$ has finite length as a $D$-module. Hence $M = \cup_{N \in \mathrm{Fin}(M)}N$. The two last conditions of Definition \ref{D3.1} are not difficult to prove. \end{proof}
Recalling that a Serre's category is a category that closes under taking submodules, quotients and extensions. If $R$ contains the rational numbers, then the category of $\Sigma$-finite $D$-modules is a Serre's subcategory of the category of $D$-module (cf. \cite[Proposition 3.7]{N14}). At the time of writing, we do not know whether the condition $\mathbb{Q} \subseteq R$ can be removed. Fortunately, the statement of Lemma \ref{L3.6} (ii) is enough for our purpose. In the following we prove the global case of \cite[Proposition 4.3]{N14}. While the proof of \cite{N14} is based on spectral sequences, our proof is elementary.
 \begin{proposition}\label{P3.7}
 Let $R$ be a (not necessary local) Noetherian ring and $S = R[x_1, ...,x_n]$ or $S = R[[x_1, ...,x_n]]$. Let $J$ be an ideal of $S$ such that $\dim R/J \cap R = 0$. Then $H^i_J(S)$ is $\Sigma$-finite for every $i \in \mathbb{N}$.
 \end{proposition}
\begin{proof} We can assume that $J$ is a radical ideal, so $J \cap R = \frak m_1 \cap \cdots \cap \frak m_r$ where $\frak m_k \in \mathrm{Max}(R)$ for all $k =  1, ..., r$. Set $J_k = \frak m_kS+J$, $k = 1, ..., r$, we have $J = J_1 \cap \cdots \cap J_r$. Since $\frak m_k + \frak m_h = R$ for all $k \neq h$, we have $J_k + J_h = S$ for all $k \neq h$. By using Mayer-Vietoris's sequence one can prove that
 $$H^i_J(S) \cong H^i_{J_1}(S) \oplus \cdots \oplus H^i_{J_r}(S)$$
 for all $i \ge 0$. Therefore, it is enough to prove the assertion in the case $J \cap R = \frak m \in \mathrm{Max}(R)$ (cf. \cite[Lemma 3.9]{N14}).
 We proceed by induction of $t = \mathrm{ht}(\frak m)$.\\
 \indent The case $t=0$, we have that $\frak m$ is a minimal prime of $R$. Let $U = H^0_{\frak m}(R)$ and $\overline{R} = R/U$. We have $U$ is an $R$-module of finite length so $U \otimes_R S$ is a $\Sigma$-finite $D$-module by Lemma \ref{L3.3}. By \cite[Corollary 3.10]{N14}, $H^i_J(U \otimes_R S)$ is $\Sigma$-finite for all $i \ge 0$. Applying local cohomology functor for the short exact sequence
 $$0 \to U \otimes_R S \to S \to \overline{S} \to 0,$$
 where $\overline{S} = \overline{R} \otimes_RS$, we get the following exact sequence
 $$\cdots \to H^{i-1}_J(\overline{S}) \to H^i_J(U \otimes_R S) \to H^i_J(S) \to H^{i}_J(\overline{S}) \to \cdots.$$
On the other hand, we have $\mathrm{Ass}_R\overline{R} = \mathrm{Ass}_RR \setminus V(\frak m)$. Notice that $\mathrm{ht}(\frak m) = 0$ so $\frak p \nsubseteq \frak m$ for all $\frak p \in \mathrm{Ass}_R\overline{R}$, and hence $\mathrm{Ann}_R(\overline{R}) \nsubseteq \frak m$. Moreover $\frak m \in \mathrm{Max}(R)$ we have $\mathrm{Ann}_R(\overline{R}) + \frak m = R$. Therefore $1 \in \mathrm{Ann}_R(\overline{R})S + J$ because $J \cap R  = \frak m$. Thus $\mathrm{Ann}_S(\overline{S}) + J = S$ since $\mathrm{Ann}_S(\overline{S}) = \mathrm{Ann}_R(\overline{R}) S$. So $H^{i}_J(\overline{S}) = 0$ for all $i \ge 0$ and hence $H^i_J(S) \cong H^i_J(U \otimes_R S)$ is $\Sigma$-finite for all $i\ge 0$.

 For $t > 0 $, set $U = H^0_{\frak m}(R)$ and $\overline{R} = R/H^0_{\frak m}(R)$. Let $\overline{S} = \overline{R} \otimes_RS$. The short exact sequence
$$0 \to U \otimes_R S \to S \to \overline{S} \to 0$$
induces the exact sequence of local cohomology modules
$$\cdots \to H^i_J(U \otimes_R S) \overset{\alpha}{\to} H^i_J(S) \overset{\beta}{\to} H^i_J(\overline{S}) \to \cdots.$$
We have the short exact sequence
$$0 \to \mathrm{im}(\alpha) \to H^i_J(S) \to \mathrm{im}(\beta) \to 0.$$
Since $U$ has finite length as an $R$-module, $U \otimes_R S$ and hence $H^i_J(U \otimes_R S)$ have finite length as a $D$-module by Example \ref{E2.1} (iv) (see also \cite[Proposition 3.3]{N13}). Thus $\mathrm{im}(\alpha)$ is a $D$-module of finite length. Suppose $H^i_J(\overline{S})$ is $\Sigma$-finite we have $\mathrm{im}(\beta)$ is also a $\Sigma$-finite $D$-module by Lemma \ref{L3.6} (i). Lemma \ref{L3.6} (ii) implies that $H^i_J(S)$ is $\Sigma$-finite for all $i \ge 0$. Therefore we can assume henceforth that $H^0_{\frak m}(R) = 0$.
Choose an $R$-regular element $a \in \frak m = J \cap R$, we have $a$ is also $S$-regular and $a \in J$. So $H^0_J(S) = 0$. For $i \ge 1$ we consider the following short exact sequence
$$0 \to S \to S_a \to S_a/S \to 0.$$
This sequence induces the exact sequence of local cohomology
$$\cdots \to H^{i-1}_J(S_a) \to H^{i-1}_J(S_a/S) \to H^i_J(S) \to H^i_J(S_a) \to \cdots.$$
Notice that $a \in J$, so $H^i_J(S_a) = 0$ for all $i \ge 0$. Thus
$$H^i_J(S) \cong H^{i-1}_J(S_a/S) \cong H^{i-1}_J(\lim_n(S/a^nS)) \cong \lim_n H^{i-1}_J(S/a^nS).$$
By inductive hypothesis we have $H^{i-1}_J(S/a^nS)$ is a $\Sigma$-finite $D(S/a^nS, R/a^nR)$-module for all $n$ and $i \ge 1$. So $H^{i-1}_J(S/a^nS)$ is a $\Sigma$-finite $D(S, R)$-module for all $n$ and $i \ge 1$. By \cite[Proposition 3.11]{N14} we need only to prove that $\cup_n \mathcal{C}(H^{i}_J(S/a^nS)$ is finite for all $i \ge 0$. We shall prove that $\mathcal{C}(H^{i}_J(S/a^nS) \subseteq \mathcal{C}(H^i_J(S/aS))$ for all $n \ge 1$. The case $n = 1$ is trivial. For $n> 1$, the short exact sequence
$$0 \to S/aS \overset{a^{n-1}\cdot}{\longrightarrow} S/a^nS \to S/a^{n-1}S \to 0$$
induces the exact sequence
$$ \cdots \to H^i_J(S/aS) \to H^i_J(S/a^nS) \to H^i_J(S/a^{n-1}S) \to \cdots.$$
Hence $\mathcal{C}(H^i_J(S/a^nS)) \subseteq \mathcal{C}(H^i_J(S/aS)) \cup \mathcal{C}(H^i_J(S/a^{n-1}S)) \subseteq \mathcal{C}(H^i_J(S/aS))$ by inductive hypothesis. The proof is complete.
\end{proof}
We are ready to prove the main result of this section, it gives a partial positive answer for \cite[Question 5.1]{N14}.
\begin{theorem}\label{T3.8} Let $(R, \frak m)$ be a local ring and $S = R[x_1, ...,x_n]$ or $S = R[[x_1, ...,x_n]]$. Let $J$ be an ideal of $S$ such that $\dim R/(J \cap R) \leq 1$. Then $H^j_{\frak mS}H^i_J(S)$ is $\Sigma$-finite for every $i, j \in \mathbb{N}$. In particular $\mathrm{Ass}_S H^j_{\frak mS}H^i_J(S)$ is finite for all $i, j \in \mathbb{N}$.
  \end{theorem}
\begin{proof}
Since $\dim R/(J \cap R) \le 1$, there exists $f \in \frak m$ such that $\frak mS \subset \sqrt{(J+fS)}$. Thus $\sqrt{J + \frak mS} = \sqrt{J + fS}$. Notice that $H^i_J(S)$ is $J$-torsion. So
$$H^j_{\frak mS}H^i_J(S) \cong H^j_{(J + \frak mS)}H^i_J(S) \cong H^j_{(J + fS)}H^i_J(S) \cong H^j_{fS}H^i_J(S)$$
for all $i, j \ge 0$. Therefore $H^j_{\frak mS}H^i_J(S) = 0$ for all $j > 1$. Hence we need only to prove that $H^0_{fS}H^i_J(S)$ and $H^1_{fS}H^i_J(S)$ are $\Sigma$-finite for all $i \ge 0$. By \cite[Proposition 8.1.2]{BS98} we have the following exact sequence
$$\cdots \to H^{i-1}_J(S) \to  H^{i-1}_J(S_f) \to H^i_{(J + fS)}(S) \to H^i_J(S) \to  H^i_J(S_f) \to \cdots.$$
On the other hand we have the following exact sequence (cf. \cite[Remark 2.2.17]{BS98})
$$0 \to H^0_{fS}H^i_J(S) \to H^i_J(S) \to  H^i_J(S_f) \to H^1_{fS}H^i_J(S) \to 0$$
for all $i \ge 0$. Therefore for each $i \ge 0$ we have the following short exact sequence
$$0 \to H^1_{fS}H^{i-1}_J(S) \to H^i_{(J + fS)}(S) \to H^0_{fS}H^i_J(S) \to 0.$$
Since $\dim R/((J + fS) \cap R) = 0$, we have $H^i_{(J + fS)}(S)$ is $\Sigma$-finite for all $i \ge 0$ by Proposition \ref{P3.7}. Hence $H^0_{fS}H^i_J(S)$ and $H^1_{fS}H^i_J(S)$ are $\Sigma$-finite for all $i \ge 0$ by Lemma \ref{L3.6}. The last assertion follows from the property of $\Sigma$-finite $D$-modules. The proof is complete.
\end{proof}
We get a result of on the finiteness of associated primes of local cohomology of polynomial rings.
\begin{corollary} \label{C3.9}
  Let $(R, \frak m)$ be a local ring and $S = R[x_1, ..., x_n]$. Let $J$ be an ideal of $S$ such that $\dim R/(J \cap R) \leq 1$. Then $\mathrm{Ass}_S H^i_J(S)$ is finite for all $i \ge 0$.
\end{corollary}
\begin{proof}
Similarly the proof of Theorem \ref{T3.8} we have an element $f \in \frak m$ such that $\frak mS \subseteq \sqrt{(J+fS)}$. Consider the exact sequence
$$\cdots \to H^i_{(J + fS)}(S) \overset{\alpha}{\to} H^i_J(S) \to  H^i_J(S_f) \to \cdots.$$
We have $\mathrm{Ass}_S H^i_J(S) \subseteq \mathrm{Ass}_S(\mathrm{im}(\alpha)) \cup \mathrm{Ass}_S H^i_J(S_f)$. Since $H^i_{(J + fS)}(S)$ is $\Sigma$-finite, so is $\mathrm{im}(\alpha)$. Hence $\mathrm{Ass}_S(\mathrm{im}(\alpha))$ is a finite set. On the other hand we have $H^i_J(S_f) \cong H^i_{(JS_f)}(S_f)$. Notice that $S_f \cong R_f[x_1, ..., x_n]$ and $\dim R_f/(JS_f \cap R_f) = 0$, we have $H^i_J(S_f)$ is a $\Sigma$-finite $D(S_f, R_f)$-module by Proposition \ref{P3.7}. So $\mathrm{Ass}_S H^i_J(S_f)$ is finite. The proof is complete.
\end{proof}

\section{Rings of dimension one}
In this section $R$ is a Noetherian ring of dimension one and $S = R[x_1, ..., x_n]$ or $S = R[[x_1, ..., x_n]]$. We recall our question.\\
{\bf Question \ref{Q1.3}.} Is it true that $H^i_J(S)$ has only finitely many associated primes for all ideals $J$ of $S$ and all $i \ge 0$?

The following is an immediate consequence of Corollary \ref{C3.9} which was shown before by N\'{u}\~{n}ez-Betancourt in \cite[Corollary 3.7]{N13}.
\begin{corollary}
Suppose that $R$ is  local and $S = R[x_1, ..., x_n]$. Then $\mathrm{Ass}_SH^i_J(S)$ is finite for all ideal $J$ and all $i \ge 0$.
\end{corollary}
We shall consider the question when $R$ contains a field of characteristic $p>0$. We start with the following.
\begin{lemma}\label{L4.2} Let $W$ is the largest ideal of finite length of $R$ and $\overline{R} = R/W$. Let $\overline{S} = \overline{R} \otimes_R S$. Suppose $\mathrm{Ass}_SH^i_J(\overline{S})$ is finite for all $i \ge 0$. Then $\mathrm{Ass}_SH^i_J(S)$ is finite for all $i \ge 0$.
\end{lemma}
\begin{proof} The short exact sequence
$$0 \to W\otimes_RS \to S \to \overline{S} \to 0$$
induces the exact sequence of local cohomology
$$\cdots \to H^i_J(W\otimes_RS) \overset{\alpha}{\to} H^i_J(S) \to H^i_J(\overline{S}) \to \cdots.$$
Since $W$ has finite length we have $H^i_J(W\otimes_RS)$ is a $\Sigma$-finite $D$-module by Lemma \ref{L3.3} and Remark \ref{R3.2}. Hence so is $\mathrm{im}(\alpha)$. Moreover $\mathrm{Ass}_SH^i_J(S) \subseteq \mathrm{Ass}_S(\mathrm{im}(\alpha)) \cup \mathrm{Ass}_SH^i_J(\overline{S})$. Therefore if $\mathrm{Ass}_SH^i_J(\overline{S})$ is finite, then so is $\mathrm{Ass}_SH^i_J(S)$.
\end{proof}

\begin{proposition}\label{P4.3} Let $R$ be an excellent domain of dimension one and of characteristic $p > 0$. Then $\mathrm{Ass}_SH^i_J(S)$ is finite for all ideal $J$ and all $i \ge 0$.
\end{proposition}
\begin{proof}
  Let $T$ be the integral closure of $R$. We have $T$ is a finitely generated $R$-module. Since $\dim R = 1$ we have $T/R$ is an $R$-module of finite length. Set $V = T \otimes_R S$. Then $V$ is either $T[x_1,...,x_n]$ or $T[[x_1,...,x_n]]$. The short exact sequence
  $$0 \to S \to V \to V/S \to 0$$
  induces the exact sequence
  $$ \cdots \to H^{i-1}_J(V/S) \overset{\alpha}{\to} H^i_J(S) \to H^i_J(V) \to \cdots.$$
Notice that $V/S$ is a $\Sigma$-finite $D$-module of finite length and so is $H^{i-1}_J(V/S)$. Therefore $\mathrm{Ass}_S (\mathrm{im}(\alpha))$ is finite. Since $T$ is Dedekind we have $V$ is a regular ring of characteristic $p>0$. So $\mathrm{Ass}_VH^i_{JV}(V)$ is finite by \cite{HS93} or \cite{L97}. By the independent theorem we have $H^i_J(V) \cong H^i_{JV}(V)$. Thus $\mathrm{Ass}_SH^i_J(V)$ is finite.
The proof is complete.
\end{proof}
The following is the main result of this section.
\begin{proposition} \label{P4.4} Let $R$ be an excellent reduced ring of dimension one and of characteristic $p > 0$. Let $S$ is either $R[x_1, ..., x_n]$ or $R[[x_1, ..., x_n]]$. Then $\mathrm{Ass}_SH^i_J(S)$ is finite for all ideal $J$ and all $i \ge 0$.
\end{proposition}
\begin{proof} By Lemma \ref{L4.2} we can assume that $\dim R/\frak p = 1$ for all $\frak p \in \mathrm{Ass}_RR$. Since $R$ is reduced, $0 = \frak p_1 \cap \cdots \cap \frak p_r$. We proceed by induction on $r$. The case $r=1$ follows from Proposition \ref{P4.3}. For $r > 1$, the following exact sequence
$$0 \to S \to (S/(\frak p_1 \cap \cdots \cap \frak p_{r-1})S) \oplus S/\frak p_rS \to S/(\frak p_1 \cap \cdots \cap \frak p_{r-1} + \frak p_r)S \to 0$$
induces the exact sequence
$$\cdots \to H^{i-1}_J(S/(\frak p_1 \cap \cdots \cap \frak p_{r-1} + \frak p_r)S) \overset{\alpha}{\to} H^i_J(S) \to H^i_J(S/(\frak p_1 \cap \cdots \cap \frak p_{r-1})S) \oplus H^i_J(S/\frak p_rS) \to \cdots.$$
Since $\frak p_1 \cap \cdots \cap \frak p_{r-1} + \frak p_r$ is not contained in any minimal prime and $\dim R=1$, we have $R/(\frak p_1 \cap \cdots \cap \frak p_{r-1} + \frak p_r)$ has finite length. Thus
$$H^{i-1}_J(S/(\frak p_1 \cap \cdots \cap \frak p_{r-1} + \frak p_r)S)$$
is $\Sigma$-finite for all $i \ge 1$. Thus $\mathrm{Ass}_S(\mathrm{im}(\alpha))$ is finite. Combining with the inductive hypothesis we obtain the assertion.
\end{proof}
Inspired by \cite{ABL05} and \cite{BBL14} we raise the following question.
\begin{question} Let $R$ be a Noetherian ring of dimension zero and of characteristic $p>0$. Let $S = R[x_1, ..., x_n]$ or $S = R[[x_1, ..., x_n]]$. For each ideal $J = (a_1,...,a_t)$ of $S$, is it true that the image of the canonical map
$$\varphi : H^i(a_1,...,a_t; S) \to H^i_J(S)$$
generates $H^i_J(S)$ as a $D$-module.
\end{question}
If the above question has a positive answer, then by the same method used in \cite{BBL14} we can extend the result of Proposition \ref{P4.4} in the case $S = R[x_1, ..., x_n]$ for any ring of dimension one and of characteristic $p >0$.

\begin{acknowledgement}
The author is grateful to the referee for his/her useful comments.
\end{acknowledgement}

\end{document}